\documentclass[12pt]{amsart}
\usepackage{amsfonts}
\usepackage{amsmath}
\usepackage{amssymb}
\usepackage{graphicx}
\usepackage{amscd}
\usepackage{xcolor}
\usepackage{hyperref}

\usepackage{alphabeta}
\usepackage{ulem}
\usepackage{cancel}

\def\mathclap#1{\text{\hbox to 0.3pt{\hss$\mathsurround=0pt#1$\hss}}}

\newtheorem{theorem}{Theorem}

\newtheorem{lemma}[theorem]{Lemma}

\newtheorem{proposition}[theorem]{Proposition}

\begin{document}

\begin{abstract}
We prove $L^{p}$-boundedness of oscillating multipliers on symmetric spaces
of noncompact type of arbitrary rank, as well as on a wide class of locally symmetric spaces.
\end{abstract}

\title[Oscillating multipliers ]{Oscillating multipliers on symmetric and
locally symmetric spaces}
\thanks{Supported by the Hellenic Foundation for Research and Innovation, Project HFRI-FM17-1733.}

\dedicatory{To the memory of Professor Michel Marias.}
\author{Effie Papageorgiou}
\email{papageoeffie@gmail.com}
\address{Department of Mathematics, Aristotle University of Thessaloniki,
Thessaloniki 54124, Greece.}
\curraddr{Department of Mathematics and Applied Mathematics, University of Crete, Voutes Campus, 70013 Heraklion, Crete, Greece.}
\maketitle

\section{Introduction and statement of the results}

The main objective in this article is to study the $L^p$ boundedness of oscillating multipliers on symmetric spaces of arbitrary rank and locally symmetric spaces. Our aim is to find the corresponding analogues of the classical Euclidean assumptions on the above mentioned geometries and to generalize the results obtained in the rank  one case by Giulini and Meda in \cite{GIUME}. The ingredients we shall use were already known, however  their present use allows us to overcome the rank obstacle in a uniform manner.
 
 {
To put the result in perspective, let us discuss the background. On $\mathbb{R}^{n}$, consider the function
	\begin{equation*}
	\widetilde{m}_{\alpha ,\beta }(\xi )=\left\Vert \xi \right\Vert ^{-\beta }e^{i\left\Vert
		\xi \right\Vert ^{\alpha }}\theta \left( \|\xi\| \right) ,\quad \alpha ,\beta >0,
	\end{equation*}%
	where $\theta$ is a smooth function, vanishing near zero and equal to $1$ outside the unit ball. As usual, denote by $C_0^{\infty}(\mathbb{R}^n)$ the set of smooth, compactly supported functions on $\mathbb{R}^n$. Let $\widetilde{T}_{\alpha ,\beta }$ be the operator which in the
	Fourier transform variables is given by
	\begin{equation}\label{multR}
	\widehat{(\widetilde{T}_{\alpha ,\beta }f)}\left( \xi \right) =\widetilde{m}_{\alpha ,\beta }(\xi )%
	\widehat{f}\left( \xi \right) \text{, \ }f\in C_{0}^{\infty }\left( \mathbb{R}%
	^{n}\right).
	\end{equation}
	In other words, $\widetilde{T}_{\alpha ,\beta }$ is a convolution operator with kernel the inverse
	Fourier transform of $\widetilde{m}_{\alpha ,\beta }$.
	This family  provides examples of operators that do not fall under the scope of Calder{\'o}n-Zygmund theory, but rather  are given by ``strongly singular kernels", \cite{FE2}. They are also interesting because of their intimate connection with the Cauchy problem for the  wave and the Schr{\"o}dinger
	equation, for $\alpha=1$ and $\alpha=2$, respectively.} In the euclidean setting, these operators have been
extensively studied, see for example \cite{FE, FE2, HIR, SCHO, STE, WAI}. The $L^{p}$-boundedness of oscillating multipliers has been studied also in
various geometric settings as Riemannian manifolds, Lie groups and symmetric
spaces, see for instance \cite{ALEX, GEORG, GIUME,MAR} and the references
therein. In particular, for the rank one case of symmetric spaces (which include hyperbolic space) and locally symmetric spaces see \cite{CGM2, GIUME, ION2} and \cite{EFFI}, respectively.

 {In the present paper we deal with oscillating multipliers in the setting of noncompact symmetric spaces of arbitrary rank. These are Riemannian, non positively curved manifolds, with a structure that induces a Fourier-like analysis.} In more detail, let $G$ be a semi-simple, noncompact, connected Lie group with finite
center and take $K$ be a maximal compact subgroup of $G$. We consider the
symmetric space of noncompact type $X=G/K$, with  {$\operatorname{dim}X=n$.} Denote by $\mathfrak{g}$ and $%
\mathfrak{k}$ the Lie algebras of $G$ and $K$, respectively. We have the
Cartan decomposition $\mathfrak{g}=\mathfrak{p} \oplus \mathfrak{k}$. Let $%
\mathfrak{a}$ be a maximal abelian subspace of $\mathfrak{p}$ and denote its dual by $\mathfrak{%
a}^*$. If $\operatorname{dim}\mathfrak{a}=d$, then we say that $X$ has rank $d$.  { The Killing form of $\mathfrak{g}$ induces a scalar product on $\mathfrak{a}$, hence on $\mathfrak{a}^*$. The norm induced by the corresponding product  on $\mathfrak{a}^*$ will be denoted by  $\|\cdot\|$.}

Let $X$ be a symmetric space of noncompact type. Consider the function
\begin{equation}
\ m_{\alpha ,\beta }(\lambda )=(\Vert \lambda \Vert ^{2}+\Vert \rho \Vert
^{2})^{-\beta /2}e^{i(\Vert \lambda \Vert ^{2}+\Vert \rho \Vert
^{2})^{\alpha /2}},\;\alpha, \beta > 0,\;\lambda \in \mathfrak{a^{\ast }},
\label{mult}
\end{equation}%
 {where $\rho$ is the half sum of positive roots counted with their multiplicity. This multiplier is the analogue of (\ref{multR}) in the present setting, but since it remains bounded for all $\lambda\in \mathfrak{a}^{\ast }$, as in \cite{GIUME}, the cut-off function in (\ref{multR}) is no longer necessary.} Denote by $\kappa_{\alpha ,\beta }$ the inverse spherical
Fourier transform of $m_{\alpha ,\beta }$ in the sense of distributions. Consider the convolution operator $T_{\alpha ,\beta }$, where
\begin{equation}
T_{\alpha ,\beta }(f)(x)= {(f\ast\kappa_{\alpha ,\beta }) (x)=}\int_{G}\kappa _{\alpha ,\beta
}(y^{-1}x)f(y)dy,\quad f\in C_{0}^{\infty }(X).  \label{operatorX1}
\end{equation}

Let $\Gamma$ be a discrete and torsion free subgroup of $G$ and let us consider the locally symmetric space $M=\Gamma \backslash X$, which equipped with the projection of the
canonical Riemannian structure of $X$, becomes a Riemannian manifold.

To define oscillating multipliers on $M$, we first observe that if $f\in
C_{0}^{\infty }(M)$, then the function $T_{\alpha ,\beta }f$ defined by (%
\ref{operatorX1}) is right $K$-invariant and left $\Gamma $-invariant. So, $%
T_{\alpha ,\beta }$ can be considered as an operator acting on functions on $%
M$, which we shall denote by $\widehat{T}_{\alpha ,\beta}$.

Let $\kappa $ be a $K$-bi-invariant function and denote by $\ast |\kappa |$ the
convolution operator whose kernel is $|\kappa |$. Let $p\in (1,\infty )$, denote by $p^{\prime }$ its conjugate and set 
\begin{equation*}
s(p)=2\min \left\{ \frac{1}{p},\frac{1}{p^{\prime }}\right\} .
\end{equation*}
We shall assume that the following version of the Kunze and Stein phenomenon holds,
\begin{equation}\label{kunzeM}
\Vert \ast |\kappa |\Vert _{L^{p}(M)\rightarrow L^{p}(M)}\leq c\underset{G}{%
	\int }|\kappa (g)|\varphi _{-i\eta _{\Gamma }}(g)^{s(p)}{dg},
\end{equation}%
where $\varphi _{\lambda }$ are the elementary spherical functions, $\eta
_{\Gamma }$ is a vector of the euclidean sphere $S(0,(\left\Vert\rho
\right\Vert^{2}-\lambda _{0})^{1/2})$ of $\mathfrak{a}^{\ast }$ and $\lambda
_{0}$ is the bottom of the spectrum of the Laplacian $\Delta _{M}$. 
For example, this is the case for $M=\Gamma \backslash G/K$, when
(i) $\Gamma $ \textit{is a lattice}, or (ii) $G$ \textit{possesses Kazhdan's property (T)} or (iii) $\Gamma \backslash G$ \textit{is non-amenable},
see \cite{LOMAjga} for more details. We say that $M$ belongs in the class (KS) if (\ref{kunzeM}), is valid on it. Note that $X$ belongs in (KS), \cite{HE, LOMAjga}.

Our main result is the following theorem. 
\begin{theorem}\label{th1} Assume that $\alpha \in (0,1)$ and that $M=\Gamma \backslash X$
belongs in the class (KS).

(i) If $\beta >n\alpha /2$, then ${T}_{\alpha ,\beta }$ (resp. $\widehat{T}
_{\alpha ,\beta }$) is bounded on $L^{p}(X)$ (resp. on $L^{p}(M)$) for all $%
p\in (1,\infty )$.

(ii) If $\beta \leq n\alpha /2$, then ${T}_{\alpha ,\beta }$ (resp. $\widehat{
T}_{\alpha ,\beta }$) is bounded on $L^{p}(X)$ (resp. on $L^{p}(M)$) for all
$p\in (1,\infty )$, provided that $\beta >\alpha n|1/p-1/2|.$

\end{theorem}

The above theorem was proved in \cite{GIUME, EFFI} for the rank one case,   {taking $\beta\in \mathbb{C}$ with $\text{Re}\beta>0$, but we will consider $\beta>0$ for simplicity}. Our proof treats symmetric and locally symmetric spaces of arbitrary rank in a uniform way. 

 {Let us make a few remarks about the symmetric space case. First, the operator $T_{\alpha, \beta}$ is bounded on $L^2(X)$ since the multiplier is bounded for all $\alpha, \beta>0$. But for $p\neq 2$, there is a  certain necessary condition (see \cite[Theorem 1]{CLERC} or \cite[p.604]{AN}), first observed by Clerc and Stein, which has no Euclidean analogue: every multiplier that yields an $L^p(X)$ bounded operator, for some $p\in (1, \infty)$, $p\neq 2$, extends to an invariant by the Weyl group, bounded, holomorphic function inside the tube $\mathcal{T}^p=\mathfrak{a^{\ast }}+i|2/p-1|C_{\rho }$. Here, $C_{\rho }$ denotes the convex hull of the images of $\rho$ under the Weyl group. In the rank one case, the interior of the tube reduces to the strip $\{\lambda \in \mathbb{C}: |\text{Im}\lambda|<|2/p-1|\rho\}$. Moreover, when $p = 1$, the multiplier should even extend to a bounded continuous function on the closed tube $\mathcal{T}^1$. Denote by $\langle \cdot, \cdot \rangle$ the $\mathbb{C}$-bilinear extension of the inner product of $\mathfrak{a}^{\ast }$ to $\mathfrak{a}_{\mathbb{C}}^{\ast }$ and observe that at $\lambda= i\rho$, the quantity $\langle \lambda, \lambda \rangle+\|\rho\|^2$ vanishes. Thus,  $m_{\alpha, \beta}$ is not defined for $\lambda= i\rho$, for any $\beta>0$ and the $L^1(X)$ problem is ill-posed.} 

 {The critical index concerning the size of $\beta$ appearing in Theorem \ref{th1} is the same as in the euclidean case. This is due to the following observation: the part of the operator ``at infinity", which is related to the large-frequencies' part of the kernel, is bounded on all $L^p(X)$, $p\in (1, \infty)$, without any restrictions on the size of the parameter $\beta$; on the other hand, the remaining ``local part" is essentially euclidean, thus inducing the condition between parameters $\alpha$, $\beta$ on Theorem \ref{th1}.  However, the result  for $\alpha=1$, cannot be obtained as a limit case of $\alpha \in (0,1)$,  $\alpha\rightarrow 1^{-}$. Indeed, for $\alpha=1$ the critical index for $\beta$ is $(n-1)\left| 1/p-1/2\right|$, rather than $n \left|  1/p-1/2\right|$}, so a different approach is required, see for instance \cite{CGM2}.

 The case $\alpha >1$ differs considerably from the corresponding Euclidean result or, for instance, the case of Riemannian manifolds of nonnegative Ricci curvature \cite[Theorem 1]{ALEX}. In fact, for $\alpha>1$, the operator is bounded only on $L^2(X)$. This is once again due to the necessary condition of Clerc and Stein. Indeed,  writing the complex number $\langle \lambda, \lambda \rangle+\|\rho\|^2$ in polar form, it is easy to see that  the multiplier $m_{\alpha, \beta}(\lambda)$ is not bounded in any tube domain $\mathcal{T}^p$, $p\neq 2$ (see also \cite[p.97]{GIUME}).
	
As usual, we perform a splitting of the kernel $\kappa _{\alpha ,\beta }$:
\begin{equation} \label{decomp}
\kappa _{\alpha ,\beta }=\zeta \kappa _{\alpha ,\beta }+(1-\zeta )\kappa
_{\alpha ,\beta }:=\kappa _{\alpha ,\beta }^{0}+\kappa _{\alpha ,\beta
}^{\infty }, 
\end{equation}%
where $\zeta \in C^{\infty }(K\backslash G/K)$ is a cut-off function such
that
\begin{equation}
\zeta (x)=%
\begin{cases}
1, & \text{if }|x|\leq 1/2, \\
0, & \text{if }|x|\geq 1.%
\end{cases}
\label{zita}
\end{equation}%
Denote by ${T}_{\alpha ,\beta }^{0}$ (resp. ${T}_{\alpha ,\beta }^{\infty }$%
) the convolution operators on $X$ with kernel $\kappa _{\alpha ,\beta }^{0}$
(resp. $\kappa _{\alpha ,\beta }^{\infty })$. Let $\widehat{T}%
_{\alpha ,\beta }^{0}$ and $\widehat{T}_{\alpha ,\beta }^{\infty }$ be the
corresponding convolution operators on $M$.  To prove the $L^{p}$ boundedness of the local part $T_{\alpha ,\beta }^{0}$
 on $X$, we follow the spectral multiplier approach of
\cite{ALEX2} (see also \cite{GEORG, MAR2}) and then use the spherical Fourier transform  in order to adapt these ideas to the symmetric space setting. Then, the result on $M$ for $\widehat{T}_{\alpha ,\beta }^{0}$ will follow. To prove the $L^{p}$ boundedness of the  parts at infinity $T_{\alpha ,\beta }^{\infty}$, $\widehat{T}_{\alpha ,\beta }^{\infty }$,  we shall make use as in \cite%
{AN, LOMAjga, EFFI} of Kunze and Stein phenomenon. We modify the proof of the main multiplier theorem in \cite{AN, LOMAjga} in order to exploit the decay rate of the derivatives of $m_{\alpha, \beta}$.

The paper is organized as follows. In Section 2 we present the necessary
tools we need for our proofs. In Section 3 we study  the $L^{p}$ boundedness
of the part of the operator near the origin. In Section 4 we treat the part
at infinity and we finish the proof of Theorem 1.

Throughout this article the different constants will always be denoted by the same letter $c$.
\section{Preliminaries}

In this section we recall some basic facts about symmetric and locally
symmetric spaces, which we will use for the proof of our results. For details see
\cite{AN, HEL2,  HEL, LOMAjga}.

Let $G$ be a semisimple Lie group, connected, noncompact, with finite center
and let $K$ be a maximal compact subgroup of $G$. We denote by $X$ the
noncompact symmetric space $G/K$. The group $G$ acts naturally on $X$ by left translations.  Denote by $\mathfrak{g}$ and $\mathfrak{k}$ the Lie algebras of $G$
and $K $ respectively. If $X$, $Y$ are two elements of $\mathfrak{g}$, then $\text{ad}(X)(Y)=[X,Y]$ is a linear transformation of $\mathfrak{g}$ to itself.  Thus, we may define the Killing form by $B(X,Y)=\text{tr}(\text{ad}X\text{ad}Y)$, which is symmetric and bilinear. Let also $\mathfrak{p}$ be the subspace of $\mathfrak{g}$ which is
orthogonal to $\mathfrak{k}$ with respect to the Killing form. 
We identify  $\mathfrak{p}$ with the tangent space at the origin $o=K$ on $X$. 

Fix $%
\mathfrak{a} $ a maximal abelian subspace of $\mathfrak{p}$ and denote by $%
\mathfrak{a}^{\ast }$ the real dual of $\mathfrak{a}$. 
The Killing form on $\mathfrak{g}$ restricts to a positive definite form on $\mathfrak{a}$. This in turn induces a positive inner product, hence a norm $\|\cdot\|$ on $\mathfrak{a}$, and by duality, on $\mathfrak{a^{\ast}}$ as well (we will use the same notation for the norms). If $\operatorname{dim}%
\mathfrak{a}=d$, we say that $X$ has rank $d$.

The Killing form endows $X$ with both a natural Riemannian metric and a corresponding $G$-invariant measure (denoted $dx$). Therefore, we can define the Laplace-Beltrami operator $\Delta _{X}$ on $X$.  If $\Gamma $ is a discrete, torsion free subgroup of $G$, then the locally
symmetric space $M=\Gamma \backslash X$, equipped with the projection of the
canonical Riemannian structure of $X$, becomes a Riemannian manifold. In the sequel we assume that $\operatorname{dim}%
X=n $.

 We say that $\alpha \in
\mathfrak{a}^{\ast }\backslash\{0\}$ is a root vector, if the space
\begin{equation*}
\mathfrak{g}^{\alpha }=\left\{ X\in \mathfrak{g}:[H,X]=\alpha (H)X,\text{
for all }H\in \mathfrak{a}\right\} 
\end{equation*}
is non-trivial. We shall denote by  $m_{\alpha}=\text{dim}\mathfrak{g}^{\alpha }$ the multiplicity of the root $\alpha$ and by $\Sigma\subset \mathfrak{a^{\ast}}$ the root system associated to $(\mathfrak{g}, \mathfrak{a})$, containing all roots. Let $W$ be the Weyl group associated to $\Sigma$, that is, the finite subgroup of isometries of $\Sigma$, generated by reflections orthogonal to the walls (the hyperplanes
orthogonal to the roots of $\Sigma$). The roots divide $\mathfrak{a}$ into Weyl chambers, maximal connected
regions where no root vanishes. 
Choose a Weyl chamber $\mathfrak{a}_{+}$, to be called
positive, and say that a root $\alpha$ is positive if $\alpha(H)$ is positive for $H$  in $\mathfrak{a}_{+}$. 
The set $\Sigma^+$ contains all positive roots and $\Sigma_0^+\subset \Sigma^+$ all  $\alpha$ that are indivisible, meaning that $\alpha/2$ is not a root. Denote by $\rho$ the half sum of positive roots counted with their multiplicities:
\[\
\rho=\frac{1}{2}\sum_{\alpha \in \Sigma^+}m_{\alpha}\alpha \in \mathfrak{a^{\ast }}.
	\]
	Thus the norm $\|\rho\|$ is defined, and the $L^2$ spectrum of the Laplace-Beltrami operator $\Delta_{X}$ consists of the half line $[\|\rho\|^2, \infty)$.

We have the Cartan decomposition on the group level by
\begin{equation}\label{kak}
G=K(\exp \overline{\mathfrak{a}_{+}})K, 
\end{equation}%
where $\overline{\mathfrak{a}_{+}}$ is the closure of the cone $\mathfrak{a}_{+}$. Let $H$ be the (unique, contrarily to the $K$ components) $\overline{\mathfrak{a}_{+}}$ component of $x \in G$ in the
decomposition (\ref{kak}) and define $|x|=\|H\|$. Viewed on $G/K$, $|x|$ is the distance of $xK$ to the origin $o=K$. Functions on $X$ are identified with the right $K$-invariant functions on $G$ and vice versa. Similarly, left $K$-invariant functions on $X$ can be viewed as $K$-bi-invariant functions on $G$.
Normalize the Haar measure $dk$ of $K$ such that $\int_{K}dk=1$. Then, from
the Cartan decomposition, it follows that
\begin{equation}\label{haar}
\int_{G}f(g)dg=\int_{K}dk_{1}\int_{{\mathfrak{a}_{+}}}{\delta (H)}%
dH\int_{K}f(k_{1}\exp (H)k_{2})dk_{2},
\end{equation}%
where the Jacobian density $\delta (H)$ satisfies
\begin{equation}\label{delta}
\delta(H)=\prod_{\alpha \in \Sigma^+}\sinh^{m_{\alpha}}\alpha(H)\asymp   \left\{ \prod_{\alpha \in \Sigma^+} \left(\frac{\alpha(H)}{1+\alpha(H)}\right)^{m_{\alpha}} \right\}e^{2\rho(H)}, 
\end{equation}
\cite[p.1038]{ANJ}, where $f(x)\asymp g(x)$ means that there exist finite positive constants $C_1\leq C_2$ such that $C_1g(x)\leq f(x)\leq C_2g(x)$. Note that if $f$ is $K$-bi invariant, then
\begin{equation}
\int_{G}f\left( g\right) dg=\int_{X}f\left( x\right) dx.  \label{rightinv}
\end{equation}

The role played by exponentials in euclidean Fourier analysis is played by the (elementary) spherical functions in the
Fourier analysis of $K$-bi-invariant functions on $G$. They are $K$-bi-invariant  and given by the integral representation
\[\varphi_{\lambda}(x)=\int_Ke^{ (i\lambda-\rho)H(xk)}dk.
\] We then have
\begin{equation}\label{sphericalphi}
|\varphi_{\lambda }(\exp H)|\leq \varphi_0(\exp H)\leq c(1+\|H\|)^ae^{-\rho(H)}, \; \lambda\in \mathfrak{a}^{\ast}, \;H\in \overline{\mathfrak{a}_{+}},
\end{equation}
for some constants $c, a>0$, \cite[p.1046]{ANJ}. Denote by $S(K\backslash G/K)$ the Schwartz space of $K$-bi-invariant functions on $G$. Then, the spherical Fourier transform $\mathcal{H}$ is defined
by
\begin{equation*}
\ \mathcal{H}f(\lambda )=\int_{G}f(x)\varphi _{-\lambda }(x)\;dx,\quad
\lambda \in \mathfrak{a^{\ast }},\quad f\in S(K\backslash G/K).
\end{equation*} Let $S(\mathfrak{a^{\ast }})$ be the usual Schwartz space on the euclidean space $\mathfrak{a^{\ast}}$ and $S(\mathfrak{a^{\ast }})^{W}$ the
subspace of Weyl-invariant functions in $S(\mathfrak{a^{\ast }})$ (e.g. radial: $m(\lambda)=m_0(\|\lambda\|)$). Then, by a
celebrated theorem of Harish-Chandra, $\mathcal{H}$ is an isomorphism
between $S(K\backslash G/K)$ and $S(\mathfrak{a^{\ast }})^W$ and its inverse
is given by
\begin{equation}\label{inversion}
\ (\mathcal{H}^{-1}f)(x)=c\int_{\mathfrak{a\ast }}f(\lambda )\varphi
_{\lambda }(x)\frac{d\lambda }{|\mathbf{c}(\lambda )|^{2}},\quad x\in
G,\quad f\in S(\mathfrak{a^{\ast }})^{W},
\end{equation}
where $\mathbf{c}(\lambda )$ is the Harish-Chandra function. It is explicitly known, but we shall only need the following rough
estimate:
\begin{equation}\label{harish}
|\mathbf{c}(\lambda )|^{-2}\leq c(1+\|\lambda\|^2)^{b/2}
\end{equation}
for some constants $c, b > 0$, \cite[p.601]{AN}.

Set
\begin{equation*}
m_{t}(\lambda )=e^{-t(\Vert \lambda \Vert ^{2}+\Vert \rho \Vert
^{2})},\quad t>0,\;\lambda \in \mathfrak{a^{\ast }}.
\end{equation*} Then the heat kernel $%
p_{t}(x)$ on $X$ is given by $(\mathcal{H}^{-1}m_{t})(x)$ \cite{ANJ}.
The heat kernel on symmetric spaces has been extensively studied, see for example \cite%
{ANJ,ANO}. Sharp estimates of the heat kernel have been obtained by Davies
and Mandouvalos in \cite{DAVMAN} for the case of real hyperbolic space,
while Anker and Ji \cite{ANJ} and later Anker and Ostellari \cite{ANO},
generalized the results of \cite{DAVMAN} to all symmetric spaces of
noncompact type.  {Recall also a few fundamental properties of the heat kernel, \cite{ANO}: it is a bi-$K$-invariant function on $G$, thus determined by its restriction to the positive Weyl chamber. Moreover it is symmetric and positive: $p_t(x, y) = p_t(y, x)> 0$, for every $x, y\in X$,  where 
\begin{equation}\label{pt(xy)}	p_t(x, y)=p_t(gK, hK)=p_t(h^{-1}g), \quad g, h\in G.
\end{equation} The heat operator $e^{t\Delta}$ is given by
\begin{equation}\label{heatop}e^{t\Delta}f(x)=\int_Xp_t(x,y)f(y)dy, \; f\in C_0^{\infty}(X),\;x\in X, \;t>0.
\end{equation} Finally, the semigroup property holds:
\begin{equation}\label{semi}\int_Xp_s(x,y)p_t(y,z)dy=p_{s+t}(x,z).
\end{equation}
}

Recall that $\Sigma _{0}^{+}$ is the set of positive indivisible roots $\alpha $
and by $m_{\alpha }$ the dimension of the root space $\mathfrak{g}^{\alpha }$. In \cite[Main Theorem]{ANO} it is proved the following sharp estimate:
\begin{align}
p_{t}(\exp {H})&\asymp ct^{-n/2}\left( \underset{\alpha \in \Sigma _{0}^{+}}{%
\prod }(1+\alpha(H) )(1+t+\alpha(H) )^{\frac{%
m_{\alpha }+m_{2\alpha }}{2}-1}\right)  \notag \\
& \times e^{-\left\Vert \rho \right\Vert ^{2}t-\rho(H)
-\left\Vert H\right\Vert ^{2}/4t},\quad t>0, \;H\in \overline{\mathfrak{a}%
_{+}},  \label{heat}
\end{align}%
where $n=$dim$X$.

From (\ref{heat}), we deduce the following crude estimate
\begin{equation}
p_{t}(\exp {H})\leq ct^{-n/2}e^{-\left\Vert H\right\Vert ^{2}/4t}, \quad
t>0, \;H\in \overline{\mathfrak{a}_{+}},  \label{ostellari}
\end{equation}%
which is sufficient for our purposes.

\section{$L^p$ boundedness of the local part}

In this section we shall prove the following proposition.

\begin{proposition}
\label{T0X} Assume that $\alpha \in (0,1)$.

\begin{enumerate}
\item[(i)] If $\beta >n\alpha /2$, then ${T}_{\alpha ,\beta }^{0}$ (resp. $%
	\widehat{T}_{\alpha ,\beta }^{0}$) is bounded on $L^{p}(X)$ (resp. on $%
	L^{p}(M)$) for every $p\in [1,\infty ]$.
\item[(ii)] If $\beta \leq n\alpha /2$, then ${T}_{\alpha ,\beta }^{0}$
(resp. $\widehat{T}_{\alpha ,\beta }^{0}$) is bounded on $L^{p}(X)$ (resp.
on $L^{p}(M)$), $p\in (1,\infty )$, provided that $\beta >\alpha n\left\vert
1/p-1/2\right\vert $.

\end{enumerate}
\end{proposition}

To prove the $L^{p}$ boundedness of the local part $T_{\alpha ,\beta }^{0}$
of the operator $T_{\alpha ,\beta }$ on $X$ we shall follow the approach of
\cite{ALEX2} (see also \cite{GEORG, MAR2}), and express the kernel $\kappa
_{\alpha ,\beta }$ of the operator $T_{\alpha ,\beta }$ via the heat kernel $%
p_{t}$ of the symmetric space $X$. 

 {As  in \cite{GIUME}, we may write
\begin{equation}\label{Tandmu}
T_{\alpha, \beta}=\mu_{\alpha, \beta}(\Delta_{X}), \text{ where }\mu_{\alpha, \beta}(\xi)=\xi^{-\beta/2}e^{i\xi^{\alpha/2}}, \; \xi>0, 
\end{equation}
and observe that
\begin{equation}\label{mandmu}
m_{\alpha, \beta}(\lambda)=\mu_{\alpha, \beta}(\|\lambda\|^2+\|\rho\|^2), \quad \lambda \in \mathfrak{a}^{\ast}.
\end{equation}}

 {Consider the functions $\omega_0, \omega \in C_{0}^{\infty }(\mathbb{R}_+)$, such that  \[\operatorname{supp}\omega_0 \subset \left\{ \xi :0\leq \xi \leq {2}\right\}, \quad \operatorname{supp}\omega \subset \left\{ \xi :1/{2}\leq \xi \leq {2}\right\}\] and take \[\omega_j(\xi)=\omega(2^{-j}\xi), \;j\in \mathbb{N}, \text{ and } \sum\limits_{j\geq 0}\omega_j(\xi)=1.\]}
	
Then, as in \cite{ALEX}, for $j\geq 0$, we write
\begin{equation}\label{muj}
\mu_j(\xi)=\mu_{\alpha, \beta}(\xi)\omega_j(\xi), 
\end{equation}
where
\begin{equation}\label{suppmuj}
\operatorname{supp}\mu_0\subset\{\xi: 0< \xi\leq 2\} \text{ and } \operatorname{supp}\mu_j\subset\{\xi: 2^{j-1}\leq \xi\leq 2^{j+1}\}, \; j\in \mathbb{N}. 
\end{equation}

 {Define the operators $T_j=\mu_j(\Delta_X)$ and note that by (\ref{Tandmu}) and (\ref{muj}), we have $T_{\alpha, \beta}=\sum\limits_{j\geq 0}T_j.$ Using the group structure, we may also write
\begin{equation}\label{Tjdef}
T_j=\ast\kappa_j=\mathcal{H}^{-1}m_j,  
\end{equation}
where
\begin{align}\label{mj}
m_j(\lambda)&=\mu_j(\|\lambda\|^2+\|\rho\|^2)=\mu_{\alpha, \beta}(\|\lambda\|^2+\|\rho\|^2)\omega_j(\|\lambda\|^2+\|\rho\|^2)\notag\\
&=m_{\alpha, \beta}(\lambda)\omega_j(\|\lambda\|^2+\|\rho\|^2), \quad \lambda\in \mathfrak{a}^{\ast}.\end{align}}
 {Observe that $m_j$ are Weyl-invariant as radial functions, so the kernels $\kappa_j$ are $K$-bi-invariant. By (\ref{suppmuj}) and (\ref{mj}), we have
\[
\operatorname{supp}m_0\subset\{\lambda:0\leq\|\lambda\|^2+\|\rho\|^2\leq 2\} \text{ and }\] \[\operatorname{supp}m_j\subset\{\lambda: 2^{j-1}\leq\|\lambda\|^2+\|\rho\|^2\leq 2^{j+1}\}, \;j\in \mathbb{N}.
\]}
 {Depending on the size of $\|\rho\|$, finitely many of the above sets may be empty. To simplify the presentation, let us assume from now on that $\|\rho\|=1$, so that}
\begin{equation}\label{suppm0}
\operatorname{supp}m_0\subset\{\lambda:\|\lambda\|\leq 2\}, \;\operatorname{supp}m_1\subset\{\lambda:\|\lambda\|\leq 4\} \text{ and } \end{equation}	
\begin{equation}\label{suppmj}
\operatorname{supp}m_j\subset\{\lambda: 2^{(j-2)/2}\leq\|\lambda\|\leq 2^{(j+1)/2}\}, \quad j\geq 2.
\end{equation}

Set
\begin{equation}
h_{j}(\xi )= {\mu}_{j}(2^{j}\ln \xi )\xi ^{-1},\quad j\geq 0,  \label{hj_function}
\end{equation}%
and observe that 
 {\begin{equation}\label{hj_supp}
\operatorname{supp}(h_{0})\subset (1,e^{2}), \quad \operatorname{supp}(h_{j})\subset (e^{1/2},e^{2}), \;j\in \mathbb{N},  
\end{equation}}
and that
\begin{align}
h_{j}(e^{2^{-j}\Delta _{X}})e^{2^{-j}\Delta _{X}}& = {\mu}_{j}(2^{j}\ln
e^{2^{-j}\Delta _{X}})e^{-2^{-j}\Delta _{X}}e^{2^{-j}\Delta _{X}}
\label{25b} \\
& = {\mu}_{j}(\Delta _{X}).  \notag
\end{align}%
 {Recall that we denoted by} $p_{t}$ the heat kernel of $X$ and by $\kappa _{j}$ the kernel of the
operator $T_{j}= {\mu}_{j}(\Delta _{X})$. Then, from (\ref{25b}) it follows that
\begin{align}
\kappa _{j}(x)& = {\mu}_{j}(\Delta _{X})\delta _{0}(x)=h_{j}(e^{2^{-j}\Delta
_{X}})e^{2^{-j}\Delta _{X}}\delta _{0}(x)  \label{2.5} \\
& =h_{j}(e^{2^{-j}\Delta _{X}})p_{2^{-j}}(x).  \notag
\end{align}
Note that
\begin{eqnarray}
\sum_{j\geq 0}\kappa _{j}(x) &=&\sum_{j\geq 0} {\mu}_{j}(\Delta _{X})\delta
_{0}(x)  \notag \\
&=& {\mu}_{\alpha ,\beta }(\Delta _{X})\delta _{0}(x)=\kappa _{\alpha ,\beta }(x).
\label{kappa}
\end{eqnarray}

Our proof will be based on the following lemma, the proof of which is postponed until the end of this section. 

\begin{lemma}
\label{kjBall} \label{kjB_1}For every $j\geq 0$,
\begin{equation*}
\Vert \kappa _{j}\Vert _{L^{1}(B_1)}\leq c2^{-\left( \beta -\alpha
n/2\right) j/2}.
\end{equation*}
\end{lemma}
Here, $B_1$ denotes the geodesic unit ball on $X$, which is a homogeneous space in the sense of Coifman and Weiss,  see  \cite{ANLO}. 
\subsection{Proof of Proposition \protect\ref{T0X}}

We treat first the case of symmetric spaces. It suffices to interpolate between the $L^{\infty}$ and $L^2$ result, and then use duality. In our proof, the properties of the spherical Fourier transform on $K$-bi-invariant functions are central. For locally symmetric spaces, the required $L^p$ boundedness for $\widehat{T}_{\alpha ,\beta }^{0}$ will follow as a consequence of the $L^p$ result for $T_{\alpha, \beta}^0$ on symmetric spaces.

(i) Let $\beta >\alpha n/2$. Recall that $\kappa _{\alpha ,\beta
}=\sum\limits_{j\geq 0}\kappa _{j}$. We shall show that $\kappa _{\alpha ,\beta
}^0\in L^1(X)$, using the fact that it is compactly supported. Indeed, by (\ref{decomp}) and Lemma \ref{kjB_1}, we have
\begin{align*}
\Vert \kappa _{\alpha ,\beta }^{0}\Vert _{L^{1}(X)} &=\Vert \zeta \kappa
_{\alpha ,\beta }\Vert _{L^{1}(X)}\leq c\Vert \kappa _{\alpha ,\beta }\Vert
_{L^{1}(B_1)} \\
&\leq c\sum_{j\geq 0}\Vert \kappa _{j}\Vert _{L^{1}(B_1)}\leq
c\sum_{j\geq 0}2^{-(\beta -\alpha n/2)j/2}<c.
\end{align*}%
This implies that
\begin{equation}
\Vert T_{\alpha ,\beta }^{0}\Vert _{L^{\infty }(X)\rightarrow L^{\infty
}(X)}\leq c.  \label{tapeiro}
\end{equation}%

It remains to show the $L^2$ result for $T_{\alpha, \beta}^0$ by summing over $T_j^0$. The spherical Fourier transform properties will allow us to estimate $m_j^0$ by the known estimates of  $m_j$. By Plancherel theorem, the $K$-bi-invariance of $\zeta$ and $\kappa_{j}$,  and  (\ref{decomp}), we get that
\begin{align}
\Vert T_{j}^{0}\Vert _{L^{2}(X)\rightarrow L^{2}(X)}& \leq \Vert
m_{j}^{0}\Vert _{L^{\infty }(\mathfrak{a^{\ast }})}=\Vert \mathcal{H}(\kappa
_{j}^{0})\Vert _{L^{\infty }(\mathfrak{a^{\ast }})}  \notag \\
& =\Vert \mathcal{H}(\zeta \kappa _{j})\Vert _{L^{\infty }(\mathfrak{a^{\ast
}})}=\Vert \mathcal{H}(\zeta )\ast \mathcal{H}(\kappa _{j})\Vert _{L^{\infty
}(\mathfrak{a^{\ast }})}  \notag \\
& =\Vert \mathcal{H}(\zeta )\ast m_{j}\Vert _{L^{\infty }(\mathfrak{a^{\ast }%
})}\leq \Vert \mathcal{H}(\zeta )\Vert _{L^{1}(\mathfrak{a^{\ast }})}\Vert
m_{j}\Vert _{L^{\infty }(\mathfrak{a^{\ast }})}.  \label{tj02}
\end{align}%
But $\zeta \in S(K\backslash G/K)$. So, as it is mentioned in Section 2, its
spherical Fourier transform $\mathcal{H}(\zeta )$, belongs in $S(\mathfrak{%
a^{\ast }})^{W}\subset L^{1}(\mathfrak{a^{\ast }})$. So,
\begin{equation*}
\Vert \mathcal{H}(\zeta )\Vert _{L^{1}(\mathfrak{a^{\ast }})}\leq c(\zeta
)<\infty .
\end{equation*}%
From (\ref{tj02}), (\ref{suppmj}) and (\ref{mult}) it follows that
\begin{align*}
\Vert T_{j}^{0}\Vert _{L^{2}(X)\rightarrow L^{2}(X)}& \leq c(\zeta )\Vert
m_{j}\Vert _{L^{\infty }(\mathfrak{a^{\ast }})} \\
& =c(\zeta )\sup_{2^{^{(j- {2})/2}}\leq \Vert \lambda \Vert \leq
2^{^{(j+1)/2}}}\left\vert m_{\alpha ,\beta }\left( \lambda \right) \omega
(2^{-j}\lambda )\right\vert \\
& \leq c(\zeta )2^{-\beta j/2},\;  {j\geq 2.}
\end{align*}
 {It is  easy to see that using the same arguments for the remaining cases $j=0,1$,  an inequality of the form above is also satisfied.} Further, by the fact that $T_{\alpha ,\beta
}^{0}=\sum\limits_{j\geq 0}T_{j}^{0}$, it follows that
\begin{eqnarray}
\Vert T_{\alpha ,\beta }^{0}\Vert _{L^{2}(X)\rightarrow L^{2}(X)} &\leq
&\sum_{j\geq 0}\Vert T_{j}^{0}\Vert _{L^{2}(X)\rightarrow L^{2}(X)}  \notag
\\
&\leq &c\sum_{j\geq 0}2^{-\beta j/2}\leq c<\infty .  \label{t2}
\end{eqnarray}%
By interpolation and duality, it follows from (\ref{tapeiro}) and (\ref{t2})
that $T_{\alpha ,\beta }^{0}$ is bounded on $L^{p}(X)$, for all $p\in
[1,\infty ]$.

(ii) Let $\beta \leq \alpha n/2$.  Once again, we shall interpolate between the $L^{\infty}$ and $L^2$ result. Recall that $T_{j}^{0}=\ast \kappa _{j}^{0}$ and that $\kappa
_{j}^{0}=\zeta \kappa _{j}$. So, from Lemma \ref{kjBall} we get that
\begin{eqnarray}
\Vert T_{j}^{0}\Vert _{L^{\infty }(X)\rightarrow L^{\infty }(X)} &\leq
&\Vert \kappa _{j}^{0}\Vert _{L^{1}(X)}=\Vert \zeta \kappa _{j}\Vert
_{L^{1}(X)}  \notag \\
&\leq &c\Vert \kappa _{j}\Vert _{L^{1}(B_1)}\leq c2^{-(\beta -\alpha
n/2)j/2}.  \label{4.5}
\end{eqnarray}%
Also, we have that
\begin{equation}
\Vert T_{j}^{0}\Vert _{L^{2}(X)\rightarrow L^{2}(X)}\leq c2^{-\beta j/2}.
\label{4.50}
\end{equation}%
Interpolating between (\ref{4.5}) and (\ref{4.50}) we get that for $p\geq 2$
\begin{align*}
\Vert T_{j}^{0}\Vert _{L^{p}(X)\rightarrow L^{p}(X)}& \leq c\Vert
T_{j}^{0}\Vert _{L^{\infty }(X)\rightarrow L^{\infty }(X)}^{1-2/p}\Vert
T_{j}^{0}\Vert _{L^{2}(X)\rightarrow L^{2}(X)}^{2/p} \\
& \leq c2^{-(1-2/p)(\beta -\alpha n/2)j/2}2^{-(2/p)\beta j/2} \\
& \leq c2^{-\left( \beta -\alpha n\left( \frac{1}{2}-\frac{1}{p}\right)
\right) j/2}.
\end{align*}%
Thus,
\begin{eqnarray*}
\Vert T_{\alpha ,\beta }^{0}\Vert _{L^{p}(X)\rightarrow L^{p}(X)} &\leq
&\sum_{j\geq 0}\Vert T_{j}^{0}\Vert _{L^{p}(X)\rightarrow L^{p}(X)} \\
&\leq &c\sum_{j\geq 0}2^{-\left( \beta -\alpha n\left( \frac{1}{2}-\frac{1}{p%
}\right) \right) j/2}<\infty ,
\end{eqnarray*}%
provided that $\beta >\alpha n\left( \frac{1}{2}-\frac{1}{p}\right) $. The $%
L^{p}$-boundedness of $T_{\alpha ,\beta }^{0}$ for $p\in (1,2)$, follows by
duality.

To prove the $L^{p}$ boundedness of the local part $\widehat{T}_{\alpha
,\beta }^{0}$ of the operator on the locally symmetric space $M$, we need the following result, \cite[Proposition 13]{LOMAanna}. 
\begin{proposition}\label{transfM}
	Assume that $p\in (1, \infty)$. If the operator $T^0=\ast\kappa^0$ is $L^p(X)$-bounded, then the operator $\widehat{T^0}=\ast\kappa^0$ is $L^p(M)$-bounded.
\end{proposition}
So, for the $L^p$ result on $M$, observe first that $\widehat{T}%
_{\alpha ,\beta }^{0}$ can be defined as an operator on the group $G$, and
then, apply the local result of Proposition \ref{T0X} to conclude its
boundedness on $L^{p}(X)$. Consequently, the continuity of $\widehat{T}_{\alpha
,\beta }^{0}$ on $L^{p}(M)$ follows by Proposition \ref{transfM}. Note that the $L^p$ boundeness of $\widehat{T}%
_{\alpha ,\beta }^{0}$ holds without any restrictions on the group $\Gamma$.

\subsection{Proof of Lemma \ref{kjB_1}}

In this section, our aim is to prove estimates of the $L^{2}$-norm of the kernels
$\kappa _{j}$, which will allow us to prove Lemma \ref{kjB_1} by using the Cauchy-Schwartz inequality.

For $r>0$, set
\begin{equation*}
V_{r}=\{H\in \mathfrak{a}:\|H\|\leq r\},\text{ and }V_{r}^{+}=V_{r}\cap
\overline{\mathfrak{a}_{+}}.
\end{equation*}%
Set also
\begin{equation}
B_{r}=\{x=k_{1}(\exp H)k_{2}\in G:k_{1},k_{2}\in K,\;H\in V_{r}^{+}\}=K\exp V_{r}^{+} K.
\label{Uj}
\end{equation}
 {The set $B_r$ consists of all points on $X$ at distance at most $r$ from the origin $K$, \cite[p.1066]{ANJ}. For small radii, observe the following euclidean upper bound for volume growth: using (\ref{haar}), (\ref{delta}) and the fact  that $\sum\limits_{\alpha\in \Sigma^+}m_{\alpha}=n-d$, \cite[p.1037]{ANJ}, we have
\begin{equation}\label{vol}
|B_r|\leq c\int_{\{H\in \mathfrak{a}_+: \|H\|\leq r \} }\prod_{\alpha\in \Sigma^+}\alpha(H)^{m_\alpha} dH\leq c r^n, \quad r\leq 1.
\end{equation}
 By $G$-invariance, the same upper bound would hold for small balls of any center (and in fact, a lower bound of the same form is true, see  \cite[p.1317]{ANLO}). 
Finally, consider the annulus}
\begin{equation*}
A_{q}=B_{2^{(q+1)/2}}\backslash B_{2^{q/2}},\quad q\in \mathbb{R}.
\end{equation*}
f
The following lemma is technical but important for the proof of Lemma \ref{8}.

\begin{lemma}\label{lemmatech}
There are constants $c>0$ and $\delta \in (0,1/8)$ such that for all $j \in
\mathbb{N},$ $q \geq -j$ and $|t|\leq\delta 2^{(q+j)/2}$,
\begin{equation}  \label{lemmacrucial}
\left| e^{ite^{2^{-j}\Delta_X}}p_{2^{-j}}(x) \right| \leq c
e^{-c2^{(q+j)/2}}2^{nj/2}, \; \text{for all }x\in A_q.
\end{equation}
\end{lemma}

\begin{proof}
By (\ref{heatop}) and (\ref{semi}), we have that
\begin{equation}
e^{ite^{2^{-j}\Delta _{X}}}p_{2^{-j}}(x)=\sum_{m\geq 0}\frac{(it)^{m}}{m!}%
e^{m2^{-j}\Delta _{X}}p_{2^{-j}}(x)=\sum_{m\geq 0}\frac{(it)^{m}}{m!}%
p_{(m+1)2^{-j}}(x).  \label{3.2}
\end{equation}%
Since $p_t$ is a $K$-bi-invariant function, the same is true for $%
e^{ite^{2^{-j}\Delta_{X}}}p_{2^{-j}}$. Bearing in mind that if $x=k(\exp H)k^{\prime}\in
A_{q}$, then $\Vert H\Vert \geq 2^{q/2}$, it follows from (\ref{3.2}) and
the estimate (\ref{ostellari}) of $p_{t}(\exp H)$ that
\begin{align}
\left\vert e^{ite^{2^{-j}\Delta _{X}}}p_{2^{-j}}(\exp H)\right\vert & \leq
c\sum_{m\geq 0}\frac{|t|^{m}}{m!}p_{(m+1)2^{-j}}(\exp H)  \notag \\
& \leq c\sum_{m\in \mathbb{N}}\frac{|t|^{m}}{m!}%
((m+1)2^{-j})^{-n/2}e^{-2^{q}/4(m+1)2^{-j}}  \label{bigsum} \\
& \leq c2^{jn/2}\sum_{m\in \mathbb{N}}m^{-n/2}\frac{|t|^{m}}{m!}%
e^{-2^{q+j}/4\left( m+1\right) }  \notag \\
& \leq c2^{jn/2}\sum_{m\in \mathbb{N}}\frac{|t|^{m}}{m!}e^{-2^{q+j}/4\left(
m+1\right) }.  \notag
\end{align}
Set
\begin{equation*}
N_{1}=\left\{ m\in \mathbb{N}:m\leq 2^{(q+j)/2}\right\}, \quad N_{2}=\mathbb{N}%
\backslash N_{1},
\end{equation*}
and
\begin{equation*}
S_{k}=\sum_{m\in N_{k}}\frac{|t|^{m}}{m!}e^{-2^{q+j}/4\left( m+1\right)
},\quad k=1,2.
\end{equation*}%
From (\ref{bigsum}), we have that
\begin{equation}
\left\vert e^{ite^{2^{-j}\Delta _{X}}}p_{2^{-j}}(\exp H)\right\vert \leq
c2^{jn/2}(S_{1}+S_{2}).  \label{S1+S2}
\end{equation}%
We shall first estimate $S_{1}$. If $m\in N_1$, then $m\leq 2^{(q+j)/2}$.
So,
\begin{equation*}
e^{-2^{q+j}/4\left( m+1\right) }\leq e^{-2^{q+j}/4\left(
2^{(q+j)/2}+1\right) }\leq e^{-2^{(q+j)/2}/8},
\end{equation*}%
and
\begin{equation*}
S_{1}\leq ce^{-2^{(q+j)/2}/8}\sum_{m\in N_{1}}\frac{|t|^{m}}{m!}%
=ce^{-2^{(q+j)/2}/8}e^{|t|}.
\end{equation*}%
But $|t|\leq \delta 2^{(q+j)/2},$ and consequently
\begin{equation}
S_{1}\leq ce^{-2^{(q+j)/2}/8}e^{\delta 2^{(q+j)/2}}\leq ce^{-c2^{(q+j)/2}},
\label{S1}
\end{equation}%
since $\delta <1/8$.

To estimate $S_{2}$, we make use of Stirling's formula: $\frac{1}{m!}\leq
c\left( \frac{e}{m}\right) ^{m}$. By the estimate (\ref{ostellari}) of $%
p_{t}(\exp H)$, and the facts that $|t|\leq \delta 2^{(q+j)/2}$ and $%
m>2^{(q+j)/2}$, we have
\begin{eqnarray*}
S_{2} &=&\sum_{m\in N_{2}}\frac{|t|^{m}}{m!}e^{-2^{q+j}/4(m+1)}\leq
\sum_{m\in N_{2}}\frac{|t|^{m}}{m!}\leq c\sum_{m\in N_{2}}(\delta
2^{(q+j)/2})^{m}\left( \frac{e}{m}\right) ^{m} \\
&\leq &c\sum_{m\in N_{2}}(\delta 2^{(q+j)/2})^{m}\left( \frac{e}{2^{(q+j)/2}}%
\right) ^{m}\leq c\sum_{m\in N_{2}}(\delta e)^{m}.
\end{eqnarray*}%
But $\delta <1/8<e^{-2}$. So,
\begin{equation}
S_{2}\leq c\sum_{m\in N_{2}}(\delta e)^{m}\leq
c\sum_{m>2^{(q+j)/2}}e^{-m}\leq ce^{-2^{(q+j)/2}}.  \label{S2}
\end{equation}%
Putting together (\ref{S1+S2}), (\ref{S1}) and (\ref{S2}), the estimate (\ref%
{lemmacrucial}) follows, and the proof of the lemma is complete.
\end{proof}

We also need the following approximation lemma, \cite{ALEX, NAT}.

For $f\in C_{0}^{k}(\mathbb{R})$, $k\in \mathbb{N}$, consider the norm
\begin{equation*}
\Vert f\Vert _{C^{k}}=\Vert f\Vert _{\infty }+\Vert f^{\prime }\Vert
_{\infty }+\dots+\Vert f^{(k)}\Vert _{\infty }.
\end{equation*}

\begin{lemma}
\label{approx} Let $f\in C_{0}^{k}(\mathbb{R})$, $k\in \mathbb{N}$ and $s
>0$. Then there exist a continuous and integrable function $\psi $ and a
constant $c>0$, independent of $s $ and $f$, such that
\begin{equation*}
\operatorname{supp}\widehat{\psi }\subset \lbrack -s ,s ],\quad \Vert \widehat{%
\psi }\Vert _{\infty }\leq c,\text{ and }\Vert f-f\ast {\psi }\Vert _{\infty
}\leq c\Vert f\Vert _{C^{k}}s ^{-k}.
\end{equation*}
\end{lemma}

Finally we need the following estimates of the functions $h_{j}$ defined in (%
\ref{hj_function}):
\begin{equation}
\Vert h_{j}\Vert _{\infty }\leq c\sup_{{\xi }\in
(e^{1/2},e^{2})}| {\mu}_{j}(2^{j}\ln \xi )\xi ^{-1}|\leq c2^{-\beta j/2},
\label{2.6}
\end{equation} 
and
\begin{equation}
\Vert h_{j}\Vert _{C^{k}}\leq c2^{-\beta j/2}2^{\alpha kj/2},\quad j,k\in
\mathbb{N}.  \label{2.7}
\end{equation}%
The proofs of (\ref{2.6}) and (\ref{2.7}) are straightforward, thus omitted.

We shall now prove the following lemma, which will allow us to prove Lemma \ref{kjB_1} by using the Cauchy-Schwartz inequality.

\begin{lemma}
\label{8} Assume that $q\leq 0$. Then, there are constants $c, c_k>0$ such that  for all $j, k\in \mathbb{N},$ $q\geq
-j $, 

\begin{enumerate}
\item[(i)] $\Vert \kappa _{j}\Vert _{L^{2}(X)}\leq c2^{-\left( \beta -\frac{n%
}{2} \right)j/2}$,

\item[(ii)] $\Vert \kappa _{j}\Vert _{L^{2}(A_{q})}\leq c_k2^{-\left( \beta -%
\frac{n}{2}+k(1-\alpha)\right)j/2}2^{-kq/2} $.
\end{enumerate}
\end{lemma}

\begin{proof}
(i) By the semigroup property of the heat operator and the estimate (\ref%
{ostellari}) of $p_{t}(x,y)$, we have that
\begin{align*}
\Vert p_{t}(\cdot ,y)\Vert _{L^{2}(X)}^{2}&
=\int_{X}p_{t}^{2}(x,y)dx=\int_{X}p_{t}(x,y)p_{t}(y,x)dx \\
& =p_{2t}(y,y)=p_{2t}(o)\leq ct^{-n/2}.
\end{align*}%
It follows that
\begin{equation}
\Vert p_{2^{-j}}(\cdot ,y)\Vert _{L^{2}(X)}\leq c\sqrt{2^{jn/2}}= c2^{jn/4},
\label{3.10b}
\end{equation}%
for any $y\in X$.

Recall now that $\kappa _{j}(x)=h_{j}(e^{2^{-j}\Delta _{X}})p_{2^{-j}}(x)$.
So, combining (\ref{3.10b}) and (\ref{2.6}) we get that
\begin{equation}
\Vert \kappa _{j}\Vert _{L^{2}(X)}\leq \Vert h_{j}\Vert _{\infty }\Vert
p_{2^{-j}}\Vert _{L^{2}(X)}\leq c2^{-\beta j/2}2^{jn/4}=c2^{-(\beta-n/2)j/2}.
\label{3.6}
\end{equation}

(ii) Let us consider a function $\psi _{j,q}$, satisfying Lemma \ref{approx}%
, i.e.
\begin{equation}  \label{3.7}
\Vert \widehat{\psi }_{j,q}\Vert _{\infty } <c,\quad \Vert h_{j}-h_{j}\ast {%
\psi }_{j,q}\Vert _{\infty }\leq \Vert h_{j}\Vert _{C^{k}}2^{-k(q+j)/2},
\end{equation}
and
\begin{equation*}
\ \operatorname{supp} \widehat{\psi }_{j,q} \subset \lbrack -\delta
2^{(q+j)/2},\delta 2^{(q+j)/2}], 
\end{equation*}
where the constant $c$ in (\ref{3.7}) is independent of $j$ and $q$.
Combining (\ref{2.7}) with (\ref{3.7}), it follows that
\begin{equation}
\Vert h_{j}-h_{j}\ast {\psi }_{j,q}\Vert _{\infty }\leq c2^{-\beta
j/2}2^{k\alpha j/2}2^{-k(q+j)/2}.  \label{convhj}
\end{equation}

Write
\begin{align}
\kappa _{j}(x)& =h_{j}\left( e^{2^{-j}\Delta _{X}}\right) p_{2^{-j}}(x)
\label{3.8} \\
& =\left( \left( h_{j}-h_{j}\ast {\psi }_{j,q}\right) +h_{j}\ast {\psi }%
_{j,q}\right) (e^{2^{-j}\Delta _{X}})p_{2^{-j}}(x).  \notag
\end{align}%
Thus
\begin{align}
\Vert \kappa _{j}\Vert _{L^{2}(A_{q})}& \leq \Vert \left( h_{j}-h_{j}\ast {%
\psi }_{j,q}\right) \left( e^{2^{-j}\Delta _{X}}\right) p_{2^{-j}}(x)\Vert
_{L^{2}(A_{q})}  \notag \\
& +\Vert (h_{j}\ast {\psi }_{j,q})(e^{2^{-j}\Delta _{X}})p_{2^{-j}}(x)\Vert
_{L^{2}(A_{q})}:=I_{1}+I_{2}.  \label{3.9}
\end{align}%
From (\ref{3.10b}), it follows that
\begin{align*}
I_{1} & \leq \Vert h_{j}-h_{j}\ast {\psi }_{j,q}\Vert _{\infty }\Vert
p_{2^{-j}}\Vert _{L^{2}(X)} \\
& \leq \Vert h_{j}-h_{j}\ast {\psi }_{j,q}\Vert _{\infty }2^{jn/4}.
\end{align*}%
But, by (\ref{convhj}),
\begin{equation*}
\Vert h_{j}-h_{j}\ast {\psi }_{j,q}\Vert _{\infty }\leq c2^{-\beta
j/2}2^{k\alpha j/2}2^{-k(q+j)/2}.
\end{equation*}
So,
\begin{equation}
I_{1}\leq c2^{-(\beta -n/2+k(1-\alpha ))j/2}2^{-kq/2}.  \label{3.10}
\end{equation}%
Let us now estimate $I_{2}$. By the inversion formula of the euclidean
Fourier transform, we have that
\begin{align}
(h_{j}\ast {\psi }_{j,q})(e^{2^{-j}\Delta _{X}})& =c\int_{\mathbb{R}}%
\widehat{(h_{j}\ast {\psi }_{j,q})}(t)e^{ite^{2^{-j}\Delta _{X}}}dt  \notag
\\
& =c\int_{\mathbb{R}}\hat{h}_{j}(t)\hat{\psi}_{j,q}(t)e^{ite^{2^{-j}\Delta
_{X}}}dt.  \label{inverse}
\end{align}%
Bearing in mind that $\operatorname{supp}\hat{\psi}_{j,q}\subset \left[ -\delta
2^{(q+j)/2},\delta 2^{(q+j)/2}\right] $, we get that
\begin{align}
\left\vert (h_{j}\ast {\psi }_{j,q})(e^{2^{-j}\Delta
_{X}})p_{2^{-j}}(x)\right\vert &\leq c\int\limits_\mathclap{{|t|\leq\delta
2^{(q+j)/2}}}|\hat{h}_{j}(t)||\hat{\psi}_{j,q}(t)| \left\vert
e^{ite^{2^{-j}\Delta _{X}}}p_{2^{-j}}(x)\right\vert dt  \notag \\
& \leq c\Vert \hat{h}_{j}\Vert _{\infty }\Vert \hat{\psi}_{j,q}\Vert
_{\infty }\int\limits_\mathclap{{|t|\leq\delta 2^{(q+j)/2}}}\left\vert
e^{ite^{2^{-j}\Delta _{X}}}p_{2^{-j}}(x)\right\vert dt.  \notag
\end{align}%
But from Lemma \ref{lemmatech} we have that
\begin{equation*}
\left\vert e^{ite^{2^{-j}\Delta _{X}}}p_{2^{-j}}(x)\right\vert \leq
ce^{-c2^{(q+j)/2}}2^{nj/2},\;\text{for all }x\in A_{q}\text{ and }\left\vert
t\right\vert \leq \delta 2^{(q+j)/2}\text{.}
\end{equation*}
So,
\begin{align}
\left\vert (h_{j}\ast {\psi }_{j,q})(e^{2^{-j}\Delta
_{X}})p_{2^{-j}}(x)\right\vert & \leq c\Vert \hat{h}_{j}\Vert _{\infty
}\Vert \hat{\psi}_{j,q}\Vert _{\infty }\int\limits_\mathclap{{|t|\leq \delta
2^{(q+j)/2}}}e^{-c2^{(q+j)/2}}2^{jn/2}dt  \notag \\
& \leq c\Vert \hat{h}_{j}\Vert _{\infty }\Vert \hat{\psi}_{j,q}\Vert
_{\infty }\delta 2^{(q+j)/2}e^{-c2^{(q+j)/2}}2^{jn/2}  \notag \\
& \leq c\Vert {h}_{j}\Vert _{1}\Vert \hat{\psi}_{j,q}\Vert _{\infty
}e^{-c2^{(q+j)/2}}2^{jn/2},  \label{convinf}
\end{align}
where in the last step we used the inequality $\Vert \hat{h}_{j}\Vert
_{\infty }\leq \Vert {h}_{j}\Vert _{1}$.

Next, recall that from (\ref{2.6}), we have that $\Vert {h}_{j}\Vert
_{\infty }\leq c2^{-\beta j/2}$. Also, by (\ref{hj_supp}), $\operatorname{supp}%
h_{j}\subset (e^{1/2},e^{2})$. These yield that
\begin{equation}
\Vert h_{j}\Vert _{1}\leq c2^{-\beta j/2}.  \label{h_jL1}
\end{equation}%
Also, from   {(\ref{3.7})}, we have that $\Vert \hat{\psi}_{j,q}\Vert
_{\infty }<c$. Combining this with (\ref{h_jL1}) and (\ref{convinf}) we
deduce that
\begin{equation}
\left\vert (h_{j}\ast {\psi }_{j,q})(e^{2^{-j}\Delta
_{X}})p_{2^{-j}}(x)\right\vert \leq c2^{-\beta j/2}e^{-c2^{(q+j)/2}}2^{jn/2}.
\label{CS}
\end{equation}%
Finally, note that if $q\leq 0$,   {then by (\ref{vol})},
\begin{equation*}
|B_{2^{(q+1)/2}}|\leq c2^{qn/2}\text{ and }A_{q}\subset B_{2^{(q+1)/2}}.
\end{equation*}%
So, by (\ref{CS}), it follows that
\begin{align*}
\Vert (h_{j}\ast {\psi }_{j,q})(e^{2^{-j}\Delta _{X}})p_{2^{-j}}\Vert
_{L^{2}(A_{q})}& \leq |A_{q}|^{1/2}\Vert (h_{j}\ast {\psi }%
_{j,q})(e^{2^{-j}\Delta _{X}})p_{2^{-j}}\Vert _{\infty } \\
& \leq c|B_{2^{(q+1)/2}}|^{1/2}2^{-\beta j/2}e^{-c2^{(q+j)/2}}2^{jn/2} \\
& \leq c2^{qn/4}2^{-\beta j/2}e^{-c2^{(q+j)/2}}2^{jn/2} \\
& \leq c2^{(q+j)n/4}2^{-\beta j/2}e^{-c2^{(q+j)/2}}2^{jn/4}.
\end{align*}%
Using that
\begin{equation*}
e^{-cx}x^{n/2}  {\leq c_{k}}x^{-k},\;c, c_{k}>0,\text{ for every }x\geq 1\text{
and }k\in \mathbb{N},
\end{equation*}%
we obtain that
\begin{align}
I_{2}& =\Vert (h_{j}\ast {\psi }_{j,q})(e^{2^{-j}\Delta
_{X}})p_{2^{-j}}\Vert _{L^{2}(A_{q})} \leq c_k2^{-\beta
j/2}2^{-k(q+j)/2}2^{jn/4}  \notag \\
&=c_k2^{-(\beta -n/2+k)j/2}2^{-kq/2}  \notag \\
& \leq c_k2^{-(\beta -n/2+k(1-\alpha ))j/2}2^{-kq/2}.  \label{3.11}
\end{align}%
From (\ref{3.11}) and (\ref{3.10}) it follows that
\begin{equation*}
\Vert \kappa _{j}\Vert _{L^{2}(A_{q})}\leq I_{1}+I_{2}\leq c_k2^{-(\beta
-n/2+k(1-\alpha ))j/2}2^{-kq/2},
\end{equation*}%
and the proof of the lemma is complete. 
\end{proof}
\subsubsection{Proof of Lemma \ref{kjB_1}}

  {Recall that by (\ref{Tjdef}) we have $\kappa_0=\mathcal{H}^{-1}m_0$. So, by (\ref{mj}), (\ref{suppm0}) and  the inversion formula (\ref{inversion}), we have
	\[
	|\kappa_0(x)|\leq \int_{\|\lambda\|\leq 2}|m_{\alpha, \beta}(\lambda)\omega_0(\|\lambda\|^2+\|\rho\|^2)\varphi_{\lambda}(x)|\frac{d\lambda}{|\textbf{c}(\lambda)|^2}.
	\]Using (\ref{sphericalphi}) and (\ref{harish}), we immediately obtain the trivial estimate
	\[|\kappa_0(\exp H)|\leq c (1+\|H\|)^ae^{-\rho(H)}, \;H\in \overline{\mathfrak{a}_{+}}. \]
	Thus, $\kappa_0$ is integrable in $B_1$, with 
	\begin{equation}\label{k0}
	\|\kappa_0\|_{L^1(B_1)}\leq c.
	\end{equation}}
   {For $j\geq 1$ write }
\begin{equation}
\int_{|x|\leq 1}|\kappa _{j}(x)|dx=\int_{|x|\leq 2^{-(1-\alpha
)j/2}}|\kappa _{j}(x)|dx+\int_{2^{-(1-\alpha )j/2}\leq |x|\leq 1}|\kappa
_{j}(x)|dx.  \label{B_1}
\end{equation}%
Using Lemma \ref{8} and the fact that by (\ref{vol}), $|B_r|\leq cr^{n},\;r\leq 1$, 
the Cauchy-Schwartz inequality implies that
\begin{align}
\int_{|x|\leq 2^{-(1-\alpha )j/2}}|\kappa _{j}(x)|dx& \leq
|B_{2^{-(1-\alpha )j/2}}|^{1/2}\Vert \kappa _{j}\Vert _{L^{2}(X)}
\label{4.1} \\
& \leq c2^{-(1-\alpha )jn/4}2^{-\beta j/2}2^{jn/4}  \notag \\
& \leq c2^{-(\beta -\alpha n/2)j/2}.  \notag
\end{align}%
Set
\begin{equation*}
{A}_{\ell }=\left\{ x\in G:2^{-(1-\alpha )(\ell +1)/2}\leq |x|\leq
2^{-(1-\alpha )\ell /2}\right\} ,
\end{equation*}%
and note that
\begin{equation*}
\left\{ x\in G:2^{-(1-\alpha )j/2}\leq |x|\leq 1\right\} \subset \cup _{\ell
=0}^{j-1}{A}_{\ell }.
\end{equation*}%
Note also that if $q=-(1-\alpha )(\ell +1)$, $\ell =0,1,...,j-1$, then, $%
0\geq q\geq -j$.

It follows that
\begin{align}
\int_{2^{-(1-\alpha )j/2}\leq |x|\leq 1}|\kappa _{j}(x)|dx& \leq \sum_{\ell
=0}^{j-1}\int_{{A}_{\ell }}|\kappa _{j}(x)|dx  \notag \\
& \leq \sum_{\ell =0}^{j-1}|{A}_{\ell }|^{1/2}\Vert \kappa _{j}\Vert _{L^{2}(%
{A}_{\ell })}  \notag \\
& \leq \sum_{\ell =0}^{j-1}|{B}_{2^{-(1-\alpha )\ell /2}}
|^{1/2}\Vert \kappa _{j}\Vert _{L^{2}({A}_{\ell })}  \label{le6a} \\
& \leq c\sum_{\ell =0}^{j-1}2^{-(1-\alpha )\ell n/4}\Vert \kappa _{j}\Vert
_{L^{2}({A}_{\ell })}.  \notag
\end{align}%
But, from Lemma \ref{8},
\begin{equation}
\Vert \kappa _{j}\Vert _{L^{2}(A_{q})}\leq c_k2^{-(\beta -n/2+k(1-\alpha
))j/2}2^{-kq/2},\text{ }  \label{le6b}
\end{equation}%
for all $j,k\in \mathbb{N},$ and $q\in \left[ -j,0\right] $.

Choosing $k>n/2$, from (\ref{le6a}) and (\ref{le6b}), it follows that
\begin{align}
\int_{2^{-(1-\alpha )j/2}\leq |x|\leq 1}|\kappa _{j}(x)|dx& \leq c\sum_{\ell
=0}^{j-1}2^{-(1-\alpha )\ell n/4}2^{-(\beta -n/2+k(1-\alpha
))j/2}2^{k(1-\alpha )\ell /2}  \notag \\
& \leq c2^{-(\beta -n/2+k(1-\alpha ))j/2}\sum_{\ell
=0}^{j-1}2^{(k-n/2)(1-\alpha )\ell /2}  \notag \\
& \leq c2^{-(\beta -n/2+k(1-\alpha ))j/2}2^{(k-n/2)(1-\alpha )j/2}  \notag \\
& \leq c2^{-j\beta /2}2^{jn/4}2^{-nj/4}2^{nj\alpha /4}  \label{4.3} \\
& \leq c2^{-(\beta -\alpha n/2)j/2}.  \notag
\end{align}
Combining (\ref{4.1}) and (\ref{4.3}), we obtain that
\begin{equation*}
\Vert \kappa _{j}\Vert _{L^{1}(B_1)}\leq c2^{-(\beta -\alpha n/2)j/2},
\end{equation*}%
and the proof of the lemma is complete.

\section{The part at infinity}

In this section we prove the $L^{p}$-boundedness of $T_{\alpha ,\beta
}^{\infty }$ (resp. $\widehat{T}_{\alpha ,\beta }^{\infty }$), which
combined with the $L^{p}$-boundedness of $T_{\alpha ,\beta }^{0}$ (resp. $%
\widehat{T}_{\alpha ,\beta }^{0}$) proved in Section 3, finish the proof of
Theorem 1.

\begin{proposition}
\label{infinityX} Assume that $\alpha \in \left( 0,1\right) $ and that $M$
belongs in the class (KS). Then the operator ${T}_{\alpha ,\beta }^{\infty }$
(resp. $\widehat{T}_{\alpha ,\beta }^{\infty }$) is bounded on $L^{p}(X)$
(resp. on $L^{p}(M)$) for all $p\in \left( 1,\infty \right) $.
\end{proposition}

For the proof of the proposition above we shall make use, as in \cite%
{AN, LOMAjga, EFFI}, of the Kunze and Stein phenomenon. We shall give the
proof of the $L^{p}$-boundedness only for $\widehat{T}_{\alpha ,\beta
}^{\infty }$. The case of $T_{\alpha ,\beta }^{\infty }$ is similar, thus
omitted. For that we need to introduce some notation.

For $p\in (1,\infty )$, set
\begin{equation}
v_{\Gamma }(p)=2\min \{(1/p),(1/p^{\prime })\}\frac{\|\eta _{\Gamma }\|}{%
\|\rho\| }+|(2/p)-1|,  \label{v(p)}
\end{equation}%
where $p^{\prime }$ is the conjugate of $p$ and $\eta _{\Gamma }\in
\mathfrak{a^{\ast }}$ is the vector appearing in (\ref{kunzeM}). Recall that $\|\eta _{\Gamma }\|=(\|\rho\|^2-\lambda_{0})^{1/2}$, where $\lambda_{0}$ the bottom of the spectrum of the Laplacian $\Delta_{M}$. Note that $v_{\Gamma }(p)\leq 1$.

Consider a bounded, Weyl invariant function $m(\lambda)$, $\lambda \in \mathfrak{a}^{\ast}$. For $N\in \mathbb{N}$, $v\in \mathbb{R}$ and $\theta >0 $, we say that a multiplier $m(\lambda)$, $\lambda \in \mathfrak{a^*}$, belongs in the class $\mathcal{M}
(v,N,\theta )$, if

\begin{itemize}
	\item $m$ is analytic inside the tube $\mathcal{T}^v=\mathfrak{a^*}
	+ivC_{\rho}$ and
	
	\item for all multi-indices $k\in \mathbb{N}$ with $|k|\leq N$, $\partial ^{k}m(\lambda )$
	extends continuously to the whole of $\mathcal{T}^{v}$ with
	\begin{equation} \label{decayrate1}
	|\partial ^{k}m(\lambda )|\leq c(1+\|\lambda\| ^{2})^{-|k|\theta /2}:=\langle\lambda
	\rangle^{-|k|\theta }. 
	\end{equation}
\end{itemize}

Let $\kappa=\mathcal{H}^{-1}m $, and denote by $\kappa ^{\infty }$ its part away from
the origin. Consider the convolution operator $\widehat{T}_{\kappa }^{\infty }$. The following result will be applied  to prove Proposition \ref{infinityX}.

\begin{proposition}\label{propos1} 
	Fix $p\in (1,\infty )$ and consider a  Weyl-invariant function $m\in
	\mathcal{M}(v,N,\theta )$, $\theta \in (0,1)$, with $v>v_{\Gamma }(p)$ and $N=\left[ \frac{n+1}{
		2\theta }\right] +1$. Then, $\widehat{T}_{\kappa }^{\infty }$
	is bounded on $L^{p}(M).$
\end{proposition}

Let us assume for the moment  that this result is valid. Then, note that 
	the multiplier $m_{\alpha ,\beta }$ belongs in the class
$\mathcal{M}(v,N, \theta)$, for $\theta =1-\alpha$. Indeed, $m_{\alpha ,\beta }\left( \lambda \right) $ has poles only at $\lambda =i\rho $ and the points in its Weyl orbit. So, the function $\lambda \longrightarrow m_{\alpha
	,\beta }\left( \lambda \right) $ is analytic in the tube $\mathcal{T}^v=\mathfrak{%
	a^{\ast }}+ivC_{\rho }$, $v\in(v_{\Gamma}(p),1)$. Secondly, for $\lambda \in \mathcal{T}^v$, it is
straightforward to see that
\begin{equation}\label{in2}
\left\vert \partial ^{k}m_{a,\beta }(\lambda )\right\vert \leq c(1+\|\lambda
\|)^{-\beta -|k|(1-a)}\leq c(1+\|\lambda \|^{2})^{-|k|(1-\alpha )/2},
\end{equation}
for every multi-index $k$. Thus, it follows from Proposition \ref{propos1} that $\widehat{T}^{\infty}_{\alpha, \beta}$ is bounded on $L^p(M)$, for all $p\in (1, \infty)$, and the proof of Theorem \ref{th1} is complete.

It remains to prove Proposition \ref{propos1}.
Since $M$ belongs in the class (KS), then, according to Kunze and Stein
phenomenon, we have that
\begin{align*}
	||\widehat{T}^{\infty }||_{L^{p}(M)\rightarrow L^{p}(M)}
	&\leq \int_{G}\left\vert \kappa ^{\infty }(x)\right\vert
	\varphi _{-i\eta _{\Gamma }}(x)^{s(p)}dx \\
	&\leq c\int_{|x|\geq 1/2}\left\vert \kappa(x)\right\vert
	\varphi _{-i\eta _{\Gamma }}(x)^{s(p)}dx.
\end{align*}%
We shall deal only with the integral for $|x|\geq 1$; the case for $1/2\leq
|x|\leq 1$ is trivial due to compactness. To estimate the integral for $%
\left\vert x\right\vert \geq 1$, we proceed as in the proof of Theorem 1 in
\cite{LOMAjga}, which is based on Proposition 5 of \cite{AN}. Using (\ref{Uj}), we have
\begin{align}
\int_{|x|\geq 1}|\kappa(x)|\varphi _{-i\eta _{\Gamma
}}(x)^{s(p)}dx& =\sum_{j\geq 1}\int_{B_{j+1}\backslash B_{j}}|\kappa(x)|\varphi _{-i\eta _{\Gamma }}(x)^{s(p)}dx\notag\\
   :&=\sum_{j\geq 1} I_j.\label{I}
\end{align}
Thus, it is sufficient to estimate the integrals in these annuli ({actually the proof in \cite{AN} uses a polyhedral variant of the balls $B_j$}).
 Set $b=n-d$, $d=\dim \mathfrak{a^{\ast }}=\operatorname{rank}X$ and let $b^{\prime }$
be the smallest integer $\geq b/2$. In \cite[p.645]{LOMAjga}, using \cite[p.608]{AN}, it is proved that   {for any $\theta>0$}, if   {$m\in \mathcal{M}(v, N, \theta)$,} then for $j\geq 1$ and every multi-index $k$ with $|k|\leq N$, we have
\begin{equation}
I_j \leq cj^{-N}j^{(d-1)/2}\notag \sum_{0\leq |k|\leq N}\left( \int_{\mathfrak{a\ast
}}(\langle \lambda \rangle^{b^{\prime }-N+|k|}|\partial _{\lambda }^{k}m(\lambda +i\rho)|)^{2}d\lambda \right)^{1/2}.  
\end{equation}\label{I_j}
We shall now modify the arguments used in the proof of \cite[Theorem 1]{LOMAjga} to control $I_j$. There, it is assumed that the decay rate of the multiplier derivatives (\ref{decayrate1}) holds true for $\theta=1$. We show that the result is still valid for $\theta \in (0,1)$ (increasing the number of  derivatives we have to control).
Let $m$ satisfy 
\begin{equation} \label{decayrate}
|\partial ^{k}m(\lambda )|\leq c(1+\|\lambda \|^{2})^{-|k|\theta
/2}:=\langle \lambda \rangle^{-|k|\theta },  
\end{equation} 
for some fixed $\theta\in (0,1)$. If $N$ is large enough, say $N>\frac{n+1}{2\theta }\geq
\frac{2b^{\prime }+d}{2\theta }$, using  (\ref{decayrate})  it follows  that
\begin{align}\label{ineq}
I_{j}& \leq cj^{-N}j^{(d-1)/2}\sum_{0\leq |k|\leq N}\left( \int_{\mathfrak{%
		a\ast }}(\langle \lambda \rangle^{b^{\prime }-N+|k|}\langle \lambda \rangle^{-|k|\theta })^{2}d\lambda
\right) ^{1/2}\notag \\
& \leq cj^{-N}j^{(d-1)/2}\left( \int_{\mathfrak{a\ast }}\langle\lambda
\rangle^{2(b^{\prime }-\theta N)}d\lambda \right) ^{1/2}\notag \\
& \leq cj^{-N}j^{(d-1)/2}.
\end{align}
Combining (\ref{I}) and (\ref{ineq}), we conclude that
\begin{equation*}
\int_{|x|\geq 1}\left\vert \kappa(x)\right\vert
\varphi _{-i\eta _{\Gamma }}(x)^{s(p)}dx=\sum_{j\geq 1}I_{j}\leq
c\sum_{j\geq 1}j^{-N}j^{(d-1)/2}<\infty ,
\end{equation*}%
since $N-\frac{d-1}{2}>1$, and the proof of Proposition \ref{propos1} is complete.

\textbf{Acknowledgment}. The author would like to thank  Professor Anestis Fotiadis for stimulating
discussions and support, as well as the anonymous referee for valuable comments and suggestions.

\end{document}